\documentclass{amsart}
\usepackage[foot]{amsaddr}
\usepackage{graphicx}
\def\cal{\mathcal}

\input{tcilatex}

\newcommand{\EE}{{\mathbb E}}

\newcommand{\NN}{{\mathbb N}}
\newcommand{\PP}{{\mathbb P}}
\newcommand{\RR}{{\mathbb R}}
\newcommand{\bP}{{\bf P}}
\newcommand{\mR}{{\cal{U}}}

\newcommand{\B}{{\cal B}}
\newcommand{\D}{{\cal D}}

\newcommand{\C}{{\cal C}}

\newcommand{\J}{{\cal J}}
\newcommand{\mH}{\cal{H}}
\newcommand{\hH}{{\widehat{\mH}}}

\newcommand{\aL}{{\cal L}}

\newcommand{\aS}{{\cal S}^{\downarrow}}
\newcommand{\aaS}{{\cal S}}

\newcommand{\V}{{\mathfrak{V}}}
\newcommand{\X}{{\cal X}}

\newcommand{\1}{{\bf 1}}

\newcommand{\mg}{\mathfrak{g}}

\newcommand{\mN}{\mathfrak{m}}
\newcommand{\omN}{\mathfrak{M}}

\newcommand{\opi}{\underline{\pi}}
\newcommand{\oPi}{\underline{\Pi}}
\newcommand{\hpi}{\widehat{\pi}}

\newcommand{\Gin}{\mathfrak{G}}
\newcommand{\oGin}{{\overline{\Gin}}}

\DeclareUnicodeCharacter{207A}{\ensuremath{^+}}  

\title[Notes]{Bayesian estimators of diversity indexes on exchangeable 
random partitions}

\author{Servet Martinez$^{1}$}
\thanks{{\it Corresponding author:} Servet Martinez, smartine@dim.uchile.cl}
\address{$^{1}$ Departamento de Ingenier\'{i}a Matem\'{a}tica \\
Centro Modelamiento Matem\'{a}tico\\
IRL 2807, UCHILE-CNRS \\
Casilla 170-3 Correo 3, Santiago, CHILE\\
E-mail: smartine@dim.uchile.cl}

\date{\today}

\begin{document}

\begin{abstract}
We study  indexes of diversity of the abundance of species when their proportions 
are organized as an exchangeable random partition and we take a sample from them. Firstly, we prove a general
result:  the sequence of posterior Bayesian estimators of any integrable function defined on countable partitions of
the unit interval is an integrable martingale that  converges a.s. and in $L^1$ to the function, when the 
sample size diverges to infinity. Hence, the posterior Bayesian estimator fluctuates as 
an integrable martingale. For the Poisson-Dirichlet Process, we study the estimators of the entropy and the 
Gini indexes in more detail.  A series of results are devoted to revealing
that the behavior of the Bayesian estimators share a number of similarities with the plug-in estimators. 
The first concerns its a.s. limit behavior, but we focus on other behaviors, expressing local relations between these estimators. 
This is the case for the one-step difference of the conditional plug-in entropy of the individuals
given that their species is known. We prove that it can be rephrased for the Bayesian  
entropy estimator and this gives a one-step difference between processes that does not jump 
only when a new species is found. Similar behavior is established for the Gini index.
\end{abstract}

\maketitle

\textbf{Running title}: Bayesian mean estimators on exchangeable random partitions.\newline

\bigskip

\noindent {\bf AMS Classification Number:} 62B10, 94A17

\bigskip 


\noindent{\bf Founding Institution:} ANID Basal PIA program FB210005 Center for Mathematical Modeling. 

\bigskip

\section{Introduction}
\label{sec1}

\smallskip

The study of the diversity of abundance of species  when the number of species and 
their abundance are unknown has been approached by using several models. For
instance, in \cite{chao2003} Shannon entropy is studied as an index of diversity and examined for classes
of data having unseen species in the sample and uses the Horvitz-Thompson adjustment
of missing species. In recent years, a number of works have taken the Poisson-Dirichlet process (PDP) 
as the prior distribution of the abundance of species.  These include 
\cite{buntine2012} in ecology, \cite{cereda} in forensic sciences, and \cite{favaro} and \cite{sharif2008}  
in machine learning. 
For this prior, the Bayesian posterior entropy  was studied in \cite{archer2012, archer2014}.
On the other, hand exchangeable random partitions of finite sets are considered in \cite{greve}
to study priors on the number and size of clusters.
 
\medskip

Under the assumption that the proportion of species 
are modeled by exchangeable random partitions we describe the behavior of  of the Bayesian posterior 
estimators of integrable functions on countable partitions of the unit interval. 
In Proposition \ref{prop1} we provide a martingale characterization of
the Bayesian estimator which serves to prove that it converge a.s. and in $L^1$.

\medskip

To be more precise: let the set $\aS$ of decreasing sequences of masses $s=(s(i): i\in \NN)$ 
of a partition of the unit interval be endowed with a probability distribution $P$. 
Consider a sequence of i.i.d. random variables uniformly distributed on $[0,1]$ 
(called observations or individuals) and let $\PP$ be their joint probability distribution.  
The observations are grouped following the classes of the partition of the interval, 
so a sample of $n$ individuals defines a partition of $\{1,..,n\}$. Let $k^n$ be
the number of classes it has and let $\pi^n(i)$ be the total number of individuals 
of the $i-$th class for $i=1,..,k^n$.  We denote by $q^n\sim \PP(\cdot | \pi^n(1),..,\pi^n(k^n))$  
the posterior distribution of the process given the number of individuals in the classes.

\medskip

In Proposition \ref{prop1} in Section \ref{sec3.1} we state one of our main results. 
Consider the filtration of $\sigma-$fields $\B_n=\sigma(\pi ^1,..,\pi^n)$, $n\in \NN$,  
and let $G:\aS\to \RR$ be a $P-$integrable function. We prove that
the sequence of Bayesian posterior means satisfies 
$E_{q^n}(G)=E_\PP(G| \B_n)$, so it is an integrable martingale, and 
\begin{equation}
\label{equ1}
\lim\limits _{n\to \infty} E_{q^n}(G)=G \hbox{ a.s. and in }L^1(\PP)
\end{equation}
holds. If $\EE(|G|^p)<\infty$ then the $L^p(\PP)$ convergence in (\ref{equ1}) also holds.

\medskip

We mention that for sum-type functions $G(s)=\sum_{n\in \NN} g(s(i))$,
the mean $E_P(G)$, and so the integrability   
is written in terms of the distribution of the first size-biased picking interval. 

\medskip

We are particularly interested in the convergence of some functions
which serve to measure the diversity of a population. We focus on the Shannon entropy $G=\mH$ and 
the Gini function $G=\Gin$ as indexes of diversity, see Section \ref{sec3.3}.  Both are of the sum-type.
Only at the end of this work, in Section \ref{sec5.4}, we make some comments on R\'enyi entropy.
For $\mH$ the integrability condition must be checked 
in order that the martingale property and (\ref{equ1}) hold. 
On the other hand, $\Gin$ is uniformly bounded so the Bayesian 
posterior estimator in (\ref{equ1}) converges for all $L^p(\PP)$, $p\ge 1$. 
We mention that in \cite{martinez2} we firstly proved the martingale property and (\ref{equ1}) 
for $G=H$.

\medskip

In Section \ref{sec2} we introduce some of the main notions for exchangeable partitions.
We closely follow the presentation in \cite{bertoin} Chapter $2$, only with some minor variations 
necessary for presenting and showing the martingale characterization. We 
use well-known time discrete-time martingale notions and properties that can be found 
in \cite{neveu}. 

\medskip

In Section \ref{sec4} we assume that  the law of the exchangeable partition is the PDP. In this
case the mean of the Shannon entropy 
$E_\bP(\mH)$ is finite. 
On the other hand  the posterior distribution $q_n$
is given in Corollary $20$ in \cite{pitman1996} and it serves as the keystone  
for computing the Bayesian posterior estimator of the entropy $E_{q_n}(\mH)$   
which was done in \cite{archer2012}. 
The second moment of the entropy and the Gini functions is computed in Section \ref{sec4.2}, 
to get an explicit bound on the second moment Doob inequality of the 
supremum of Bayesian posterior estimators. Since it is finite 
the martingale difference for the entropy and the Gini indexes 
cannot satisfy a Central Limit Theorem as in \cite{ibra}, even if the martingale difference is
uniformly bounded. For the entropy the uniform boundedness property is shown in 
Proposition \ref{prop2}. 

\medskip

In Section \ref{sec5} we seek to study the relations between the Bayesian posterior estimators and 
the plug-in estimators for exchangeable random partitions. This is done for the entropy and the Gini indexes
in two situations. The first consists of the limit behavior, which is studied in Section \ref{sec5.1}. When
$E_\bP(\mH)$ is finite, it follows directly from \cite{antos} that the plug-in entropy estimator 
$\hH_n=\mH(\pi^n/n)$ satisfies $\lim\limits_{n\to \infty} \hH_n=\mH\, \hbox{ a.s.}$. 
A similar relation holds for the Gini function. Thus, both the Bayesian posterior and the plug-in estimators converge 
a.s. to the entropy and the Gini indexes. In particular, this result extends relation
$|\hH_n-E_{q^n}(\mH)|\to 0$ in probability as $n\to \infty$, shown in \cite{archer2014} 
for the entropy in the PDP case, whose proof makes heavy use of the explicit expression of the Bayesian posterior mean.  

\medskip

Further, in Sections \ref{sec5.2} and \ref{sec5.3}, in the PDP frame, we state other kind of relations 
between  the Bayesian posterior and the plug-in estimators related to
the evolution in one step of time and the maxmin properties of these estimators at some fixed time.
For instance,  with the plug-in estimator, one computes the one-step difference of the conditional entropy 
of individuals given that their species are known. This quantity increases in time and it remains constant 
only when a new species is found. This phenomenon is retrieved for the Bayesian posterior estimators 
of the entropy and Gini indexes, providing a one-step difference between increasing adapted 
processes that remains constant only at the times of new-species discovery, see Proposition \ref{prop4}. 
This result is summarized in Corollary \ref{cor1}.
These properties deepen and enlarge previous results done in \cite{martinez} for 
the entropy function. 

\medskip

From this presentation it should be clear that the references \cite{pitman1996, archer2012, archer2014}
and \cite{martinez}, form the basis of our results in Section \ref{sec4}, and Sections \ref{sec5.2} and \ref{sec5.3}, 
respectively.  

\smallskip

\section{Exchangeable partitions}
\label{sec2}

\smallskip

A random partition $\Xi$ of the set of integers $\NN=\{1,2,...\}$ is exchangeable if its
law is invariant under the class of permutations 
$\V=\{\varphi:\NN\to \NN : \exists N(\varphi),  \varphi(i)=i \hbox{ for } i>N(\varphi) \}$. The invariance in law 
means $\varphi(\Xi)\sim \Xi$ $\; \forall \varphi\in \V$, where the classes of $\varphi(\Xi)$ are described by: 
the class of  $\varphi(\Xi)$ containing $i\in \NN$ is $\varphi(\Xi)(i)=\varphi^{-1}(\Xi(i))$. 
The Kingman's theory describe this class of partitions by using random partitions of 
the unit interval $[0,1]$. Let us give a brief presentation of it.

\medskip

Let $\aS=\{s=(s(i): i\in \NN): s(i\!+\!1)\ge s(i)\!\ge \!0\;  \forall i\in \NN, \, \sum_{i\in \NN}s(i)=1\}$ be the 
class of probability sequences ordered in a decreasing way, be endowed  with a probability measure
$P$. Let $S$ be a random sequence with values in $\aS$ and distributed 
according to  $P$, this is denoted by $S\sim P$. 
The mean expected value with respect to $P$ is denoted by $E_P$.

\medskip

To each $s\in \aS$ one associates a collection of disjoint open intervals 
$\J^s=(J^s(i): i\in \NN)$ of $[0,1]$ such that $s=(|J^s(i)|^{\downarrow}: i\in \NN )$ is the sequence of  
interval lengths ranked in a decreasing way. 
Then, $\J^S$ is a collection of open intervals associated to $S$, called  the classes of $S$.

\medskip

Let $\X=(X_l: l\ge 1)$ be a sequence of i.i.d. Uniform r.v's in $[0,1]$ independent of $S$. Its law is noted $\mR$
and it is invariant under $\V$, that is  $\varphi(X)\sim \mR$ for all $\varphi\in \V$, where
$\varphi(\X)=(X_{\varphi(i)}: i\in \NN)$. The random elements $S,\X$ are independent so its joint 
law  is the product probability $\PP=P\otimes \mR$ .
For a random element $Y$ depending on $(S, \X)$, $\PP(\cdot | Y=y)$ denotes
the conditional distribution given $Y=y$. 

\medskip

Consider the sequence of species $\X^*=\X^*(S,\X)=(X^*_n: n\ge 1)$ 
that is recursively defined by: 
$X^*_1=1$ and if, up to $n$, $\{X_1,...,X_n\}$ have visited $k^n$ different intervals in $\J^S$, then 
$X^*_{n+1}\!=\!j$ and $k^{n+1}\!=\!k^n$ if   
$X_{n+1}$ is in the same interval as $X^*_j$ for some $j\!\in \!\{1,..,k^n\}$; and
$X^*_{n+1}\!=\!k^n\!+\!1$ and $k^{n+1}\!=\!k^n\!+\!\1$, if 
$X_{n+1}$ belongs to an interval that has not been visited before $n$.
So, the classes are numbered in $\NN$ sequentially as  they are encountered. 
Most of the notions will only depend on the probability law of the pair $(S,\X^*)$, which is denoted 
by $\bP$. 

\medskip

Let $\xi^n=(\xi^n(1),..,\xi^n(k^n))$ be the partition of $\{1,..,n\}$ with classes 
$\xi^n(i)=\{j\in \{1,..,n\}: X^*_j=i\}$. 
The sequence of partitions $(\xi^n: n\in \NN)$ is compatible, this means $\xi^{n+1}|_{\{1,..,n\}}=\xi_n \; \forall n$, 
and it defines an exchangeable random partition $\Xi=(\xi(i): i\in \NN)$ of $\NN$. 
Let $\pi^n(j)=\# \xi^n(j)$ be the number
of elements of $\xi^n(j)$ and define the vector of abundance of the species, $\pi^n=(\pi^n(1),..,\pi^n(k^n))$.
One has $k_1=1$, $\pi^1=(1)$ and 
$$
\pi^{n+1}= 
\begin{cases}
&\!\! \pi^n\!+\delta_{k^n}(j) \hbox{ if }X^*_{n+1}=j \hbox{  for some } j\in \{1,..,k^n\},  \\
&\!\! (\pi^n,1) \hbox{ if } X_{n+1}\not\in \{X^*_1,..,X^*_n\};
\end{cases}
$$
where $\delta_k(j)$ is the $k-$dimensional vector with all $0'$s except by a $1$ in position $j$.
Consider the sequence $\Pi=(\pi^n: n\in \NN)$. Notice that given $S$, the information of $\X^*$ and
$\Pi$ determine each other, because $\X^*$ defines $\Pi$, and reciprocally  the passage from 
$\pi^n$ to $\pi^{n+1}$ determines the value of $X^*_n$. Then, instead of the pair $(S,\X^*)$ we often consider 
$(S,\Pi)$. So, $\bP$ is the distribution of the pair  $(S,\Pi)$ and $\bP(\cdot | s)$ denotes 
 the conditional distribution given $S=s$. 
 
\medskip

In Kingman theory it is shown there exist the asymptotic frequencies 
$\hpi(i)=\lim\limits_{n\to \infty} \pi^n/n$ $\,\bP( \cdot | s)-$a.s., 
and the vector $\hpi=(\hpi(i): i\in \NN)$ is a size-biased 
reordering of $s$, see \cite{kingman} and also \cite{pitman1995}. 
Then, the decreasing ranked sequence $\hpi^{\downarrow}$ is distributed with
law $P$ in $\aS$. Moreover, the  law of any exchangeable random partition $\Xi$ 
can be set in the form $d\bP(s,\Pi)=\bP(\Pi |s) dP(s)$, see Theorem $2.1$ in \cite{bertoin}.

\medskip

As seen, the construction of exchangeable partitions depends on the law $P$ on $\aS$.
The law can be defined by an inhomogeneous Poisson 
process on $(0,\infty)$ having a.s. a finite 
number of points in $(1,\infty)$ and an infinite number of points in $(0,1)$. 
See Section 2.2. in \cite{bertoin} and Section $3$ in \cite{pitman2003}. 
Meaningful examples of laws $P$ are given in Section $5$ of  \cite{pitman2003}. 

\smallskip

\section{Martingale characterization}
\label{sec3}

\smallskip

Let $\Pi_n=(\pi^r: r\le n)$, so $d\bP(s,\Pi_n)=\bP(\Pi_n | s) dP(s)=\bP(ds | \Pi_n)\bP(\Pi_n)$ 
where  $\bP(\cdot | \Pi_n)$ is the conditional distribution given $\Pi_n$. 
We also note $\bP(\cdot | \pi_n)$ the conditional distribution given $\pi_n$. 

\medskip

By $\Pi$ we also mean a value taken by $\Pi(s,\X)$, so $\Pi_n$ is the
vector containing the first $n$ coordinates of $\Pi$ and $\pi^n$ is its $n-$th coordinate. 
We shall note by $\omN$ the set of all values $\Pi$, by  $\omN_n$ the set of values 
$\Pi_n$ and by  $\mN_n$ be the set of all values $\pi^n$.

\medskip

Let $\aaS=\{s=(s(i): i\! \in \NN): s(i)\!\ge 0\;  \forall i\!\in \NN, \sum_{i\in \NN}s(i)=1\}$ and take
a function $G:\aaS\to \RR$ symmetric in the order of the components $(s(i): i\in \NN)$.
Thus, $G(s)=G(s')$ where $s'\in \aS$ is the sequence of the components of
$s$ ranked in a decreasing way. 

\smallskip

\subsection{Bayes posterior estimators convergence}
\label{sec3.1}

\smallskip

Let $\B_n=\sigma(\Pi_n)$  be the $\sigma-$field generated by the random element $\Pi_n$, 
so $\B_n$  increases with $n\in \NN$. 

\begin{proposition}
\label{prop1}
Let $G:\aS\to \RR$ and assume $E_P(|G|)<\infty$. Let $E_{\bP(\cdot| \pi^n)}(G)$ be 
the Bayes posterior mean of  $G$ at step $n$. One has,
\begin{equation}
\label{equ2}
\forall n\in \NN: \; E_{\bP(\cdot| \pi_n)}(G)=E_\bP(G|\B_n)(\Pi),
\end{equation}
and $\left(E_{\bP(\cdot| \pi_n)}(G): n\in \NN\right)$ is an integrable $\bP-$martingale with 
respect to the filtration $(\B_n: n\!\in \!\NN)$. We have
\begin{equation}
\label{equ3}
\lim\limits_{n\to \infty}E_{\bP(\cdot| \pi^n)}(G)=G \; \; \bP\!-\!\hbox{a.s. and in } L^1(\bP),
\end{equation}
and if $E_P(|G|^p)<\infty$ for some $p>1$, then the limit holds in $L^p(\bP)$.
\end{proposition}

\begin{proof}
Firstly, let us check the following equality of conditional laws,
\begin{equation}
\label{equ4}
\bP(ds, \Pi_{n+1} | \Pi_n)=\bP(ds, \Pi_{n+1} | \pi^n). 
\end{equation}

We note $\X_n=(X_1,..,X_n)$. Take $\opi^n\in \mN_n$ and let
$\omN_n(\opi^n)\!\!=\!\!\{\Pi_n\!\in \! \omN_n\!: \!\pi^n\!\!=\!\!\opi^n\}$. 
For $\Pi_n^1$, $\Pi_n^2\in \omN_n(\opi^n)$ 
there is a permutation $\varphi_n$ of $\{1,..,n\}$ 
satisfying $\{\Pi_n(s,\X_n)=\Pi_n^1\}=\{\Pi_n(s,\varphi_n(\X_n))=\Pi_n^2\}$. 
Since the law of $\X$ is $\V-$invariant one has, 
$$
\PP(\Pi_{n+1},\! \Pi_n(\!s,\! \X_n)\!\!=\!\Pi_n^1| s)
\!=\! \PP(\Pi_{n+1}, \! \Pi_n(\!s,\! \X_n))\!=\!\Pi_n^2|s).
$$
From $\PP(ds, \! \Pi_{n+1},\! \Pi_n(\!s,\! \X_n)\!\!=\!\Pi_n^i)
=\PP(\Pi_{n+1},\! \Pi_n(\!s,\! \X_n)\!\!=\!\Pi_n^i| s)P(ds)$ for $i=1,2$ one gets
$\PP(ds, \Pi_{n+1} |\Pi_n(s,\X_n)=\Pi_n^1)=\PP(ds, \Pi_{n+1} | \Pi_n(s,\X_n)=\Pi_n^2)$ 
and  (\ref{equ4}) follows.

\medskip
 
Since $d\bP(s,\Pi_n)=\bP(ds | \Pi_n)\bP(\Pi_n)$, every measurable function $\mg_n:\aS \times {\omN}_n\to \RR$ 
that is nonnegative or  ${\bf P}-$integrable,  satisfies
$$
\int\limits_{\aS}\! \sum\limits_{\Pi_n\in \omN_n}  \!\!\!\! \mg_n(s,\Pi_n) 
\bP(\Pi_n | s) dP(s)=\!\int \!   \! \mg_n d\bP\!=\!
\sum\limits_{\Pi_n\in \omN_n}\int\limits_{\aS} \!\! \mg_n(s,\Pi_n) \bP(ds | \Pi_n) \bP(\Pi_n).
$$
(See Lemma $1$ in \cite{ishwaran2003}). 
Now, (\ref{equ4}) gives $\bP(ds | \Pi_n)\bP(\Pi_n)=\bP(ds | \pi^n)\bP(\Pi_n)$, so we get
\begin{equation}
\label{equ5}
\int\limits_{\aS}\! \sum\limits_{\Pi_n\in \omN_n}  \!\!\!\! \mg_n(s,\Pi_n) 
\bP(\Pi_n | s) dP(s)\!=\!\int \!   \! \mg_n d\bP\!=\!\!\!\!
\sum\limits_{\Pi_n\in \omN_n} \int\limits_{\aS} \!\! \mg_n(s,\Pi_n) \bP(ds | \pi^n) \bP(\Pi_n).
\end{equation}

Fix some value $\oPi_n=(\opi^j: j=1,..,n)\in \omN_n$ and 
take the integrable function $\mg_n(s,\Pi_n)= G(s) {\bf 1}(\Pi_n(s,\X_n)=\oPi_n)$ in formula (\ref{equ5}). Then,
$$
\int_{\aS} G(s) \bP(\oPi_n | s)P(ds)
=\bP(\oPi_n)\int_{\aS} G(s) \bP(ds | \opi^n)=\bP(\oPi_n) E_{\bP(\cdot| \opi^n)}(G).
$$
Since 
$$
\int_{\aS} G(s) \bP(\oPi_n | s)P(ds)=\int_{\aS} G(s) E({\bf 1}_{\Pi_n(s,\X_n)=\oPi_n} | s)P(ds)
=\int_{\{\Pi_n(s,\X_n)=\oPi_n\}} \! G \, d\bP ,
$$
we get
$$
E_{\bP(\cdot| \opi^n)}(G)=\frac{1}{\bP(\Pi_n(s,\X_n)=\oPi_n)}\int_{\{\Pi_n(s,\X_n)=\oPi_n\}} G d\bP=
E_\bP(G|\B_n))(\oPi_n),
$$
which shows (\ref{equ2}):
$\left(E_{\bP(\cdot| \pi_n)}(G): n\in \NN\right)$ is an integrable $\bP-$martingale with 
respect to the filtration $(\B_n: n\!\in \!\NN)$.
The limit $\sigma-$field is $\B_\infty=\sigma(\Pi)$ $\; \bP-$completed. 
We have $\lim\limits_{n\to \infty} E_{\bP(\cdot| \pi_n)}(G)=G$ $\, \bP-$a.s. 
because the $\sigma(\Pi)-$measurable sequence $\hpi=\lim\limits_{n\to \infty }\pi^n/n$ 
is a size-biased reordering of $s$, $\; \bP(\cdot | s)$ a.s.,
and since $G$ is symmetric in its components one gets $G(\hpi)=G(s)$ $\; \bP(\cdot | s)$ a.s. 
Then, the $\bP-$a.s. convergence in (\ref{equ3}) is satisfied.

\medskip

The martingale theorem for integrable martingales 
gives the $L^1(\bP)$ convergence in (\ref{equ3}) and when  $E_P(|G|^p)<\infty$ for some 
$p>1$, then it follows the $L^p(\bP)$ convergence in (\ref{equ3}) (see Proposition 
$II\!-\!2\!- \!11$ in \cite{neveu}). So, (\ref{equ3}) is satisfied. 
\end{proof}

\subsection{Comments derived from the martingale property}
\label{sec3.2}

Since $\pi^1$ is constant, $\B_1=\sigma(\pi^1)$ is trivial and $E_{\bP(\cdot|\pi^1)}(G)=E_\bP(G)$.
On the other hand $\bP(\Pi_1=(1))=\bP(\pi^1=1)=1$, and from  (\ref{equ4}), 
$\bP(\Pi_{n+1} | \Pi_n)=\bP(\pi^{n+1}| \pi^n)$. Then, 
\begin{equation}
\label{equ6}
\bP(\Pi_n)=\prod_{k=1}^{n-1} \bP(\pi^{k+1}| \pi^{k}).
\end{equation}

\medskip

Let $\pi^n\in \mN_n$. Then,  
$\C(\pi^n)=\{\pi^n+\delta_{k^n}(j): j\!=\!1,..,k^n\}\cup  \{(\pi^n,1)\}$
is the set of the values that can take $\pi^{n+1}$. Since 
$E({\bf 1}_{\pi^{n+1}}|\B_n)(\pi^n)=\bP(\pi^{n+1} | \pi^n)$, 
the martingale property  
$E_{\bP(\cdot| \pi^{n})}(G)=E(\,E(G|\B_{n+1})\, |\B_n)(\pi^n)$ gives,
$$
\forall n\in \NN:\; E_{\bP(\cdot| \pi^{n})}(G)=\sum_{\pi^{n+1}\in \C(\pi^n)} 
E_{\bP(\cdot| \pi^{n+1})}(G)\bP(\pi^{n+1}|\pi^n).
$$
If $V_{n+1}$ as a $\B_{n+1}$-measurable function and $h$ a nonnegative Borel function then,
\begin{equation}
\label{equ7}
E(h(V_{n+1})|\B_n)(\pi^n)=\sum_{\pi^{n+1}\in \C(\pi^n)} 
h(V_{n+1}(\pi^{n+1})) \bP(\pi^{n+1}|\pi^n).
\end{equation}

\smallskip

If $G\in L^2(P)$ the Doob maximal inequality (see Proposition IV-2-8 in \cite{neveu}) applied to 
the martingale $(E_{\bP(\cdot |\pi^n)}(|G|: n\in \NN)$ gives,
\begin{equation}
\label{equ8}
E_{\bP}\left(\left(\sup\limits_{n\in \NN} E_{\bP(\cdot |\pi^n)}(|G|)\right)^2\right)
\le 4 \sup\limits_{n\in \NN}E_\bP\left(E_\bP\left(|G| \, | \B_n\right)^2\right)\le 4 E_\bP\left(|G|^2\right).
\end{equation}

\medskip

Finally, notice that one can randomize the Bayes posterior mean.
For instance take the sequence of times when new species appear, 
$$
\tau^1=1, \; \tau^{n+1}=\inf\{t> \tau^n: \pi^{t+1}=(\pi^t,1)\}, n\ge 1,
$$
which are finite a.s. stopping times and $\PP(\lim\limits_{n\to \infty}\tau^n=\infty)=1$. 
From (\ref{equ2}),
$$
E_\bP(G|\B_{\tau_n})(\pi^{\tau_n})=\sum_{r\in \NN}  \1_{\{\tau_n=r\}} E(G|\B_r)(\pi^r)=E_{\bP(\cdot |\pi^{\tau_n})}(G).
$$
So, $(E_{\bP(\cdot |\pi^{\tau_n})}(G): n\in \NN)$ is an integrable martingale with respect to
the filtration of $\sigma-$fields $(\B_{\tau^n}: n\in \NN)$ and it 
converges a.s. and in $L^1(\bf P)$, see Corollary $IV\!-\!2\!-\!6$ in \cite{neveu}.

\medskip

\subsection{The entropy and the Gini indexes}
\label{sec3.3}

We are interested on functions $G$ which serve to measure the degree of 
mixture of proportions of the species in some population. In particular we will focus on: 

\smallskip

The Shannon entropy $\mH(s) = - \sum_{i\in \NN}s(i) \log(s(i))$, introduced and studied in \cite{shannon} and
the Gini function $\Gin(s)=1-\sum_{i\in \NN} s(i)^{2}$, that is the probability that 
two independent classes, both chosen with distribution $s$, are different.
It was introduced in \cite{gini} to study the diversity of groups in a population. 
We also consider $\Gin^{(\kappa)}(s)=1-\sum_{i\in \NN} s(i)^{\kappa+1}=\sum_{i\in \NN}s(i)(1-s(i)^\kappa)$ 
the generalized Gini function of parameter $\kappa>0$. When $\kappa\in \NN$ it is
the probability that $\kappa+1$ independent classes chosen with law $s$, 
are not all the same.

\medskip

The generalized Gini function and the Shannon entropy 
satisfy  the hypotheses of what is called an impurity function defined on the sequences $s\in \aaS$ 
having a finite set of non-vanishing masses (see Definition $2.5$ and Proposition $A.1$ in \cite{breiman}).
On the set of $s\in \aaS$ having at most $n$ non-vanishing masses these hypotheses are: 
being nonnegative and symmetric in the components; vanishing at $s=\delta_n(i)$; 
reaching its maximum at $(1/n,..,1/n)$; 
and being concave, so the function of a mixture of sequences of masses is greater or equal
than the mixture of the functions of these sequences.

\medskip

The indexes $\mH$ and $\Gin^{(\kappa)}$ are sum-type functions, that is of the form 
$G(s)=\sum_{i\in \NN} g(s(i))$ with $g:[0,1]\to \RR_+$ a Borel function. For this type of 
functions it holds,
$$
E_P\left(\sum_{i\in \NN} g(S(i))\right)=\int_0^1 \frac{g(x)}{x} dF(x)
\hbox{ where }\hpi(1)\sim F,
$$ 
that is $F$ is the distribution of the first size-biased picking interval. 
See relation ($25$) in \cite{pitman2-1996}. 
For $G=\mH$ one has $E_P(\mH)=-\int_0^1 \log(x)dF(x)$,
so $E_P(\mH)\!< \!\infty$ is equivalent to $\int_0^1 {\log}(x)dF(x)\!> \!-\infty$ and
when this holds we can apply Proposition \ref{prop1} to $\mH$.
On the other hand, we have 
$E_P(\Gin^{(\kappa)})=1-\int_0^1 x^\kappa dF(x)$. Moreover, since $\Gin^{(\kappa)}$ 
is bounded, $E_P((\Gin^{(\kappa)})^p)$ is finite $\forall p\ge 1$ so in  
Proposition \ref{prop1} the $L^p(\bP)$ convergence is satisfied.

\medskip

 \section{Poisson Dirichlet Process}
\label{sec4}

\smallskip

An important class of exchangeable partitions is given 
by the two parameter PPD, denoted  PDP$(\alpha,\theta)$, 
with $0\leq \alpha< 1$ and $\theta>-\alpha$, 
introduced in \cite{pitman1997}. When it is necessary we denote its law by ${\bf P}= 
{\bf P}^{\alpha,\theta}$.
In  \cite{pitman2003} its  size-biased sequence  $\hpi=(\hpi(i): i\in \NN)$ 
is characterized by the independence property of the variables 
$(W_j: j\in \NN)$ defining $\hpi$ through:  $\hpi(1)= W_1$ and 
$\hpi(j)= W_j \prod_{i=1}^{j-1}(1-W_i)$ for $j\geq 2$. They
are distributed as $W_j\sim \, \text{Beta}(1-\alpha, \theta +\alpha j)$ for $j\ge 1$. 

\medskip

In this Section we assume that $G:\aS\to \RR$ is of the sum-type $G(s)=\sum_{i\in \NN}g(s(i))$
with $g\ge 0$.
Denote by $e_{(a,b)}(h(X))$ the mean of $h(X)$ when $X\sim \hbox{Beta}(a,b)$.
Then, 
$E_{{\bf P}^{\alpha,\theta}}(G(S))=e_{(1-\alpha,\theta+\alpha)}\left(\frac{g(X)}{X}\right)$. 

\medskip

Let $\psi(x)=\Gamma'(x)/\Gamma(x)$ be the digamma function, which is strictly increasing on $(0,\infty)$. 
One has $e_{(a,b)}(-\log X)=\psi(a+b)-\psi(a)$, $e_{(a,b)}(X)=a/(a+b)$ and more generally 
$e_{(a,b)}(X^m)=\prod_{r=0}^{m-1}\frac{a+r}{(a+b+r)}$  for $m\in \NN$. Then, 
$$
E_{{\bf P}^{\alpha,\theta}}(\mH)=\psi(\theta+1)-\psi(1-\alpha),\; 
E_{{\bf P}^{\alpha,\theta}}(\Gin)=\frac{\theta+\alpha}{\theta+1},\; 
E_{{\bf P}^{\alpha,\theta}}(\Gin^{(\kappa)})=1-\prod_{r=0}^{\kappa}\frac{(1-\!\alpha+\!r)}{(\theta\!+1+\!r)}\,,
$$
when $\kappa\in \NN$. 
  
\subsection{Bayesian posterior means}
\label{sec4.1}

In \cite{pitman1996} formulae $(42)$ and $(33)$, it is shown that the transition kernel 
$\bP(\Pi_{n+1}| \Pi_n)=\bP(\pi^{n+1}|\pi^{n})$ in  (\ref{equ6})
is given by the Pitman formula, 
\begin{equation}
\label{equ9}
{\bf P}^{\alpha,\theta}(\pi^n\!+\!\delta_{k^n}(j)|\pi^n)\!=\!
\frac{\pi^n(j) \!-\! \alpha}{\theta \!+\! n}, \, j=1,..,k^n \hbox{ and }
\bP((\pi^n,1)| \pi^n) 
\!=\! \frac{\theta \!+\! \alpha k^n}{\theta\!+ \! n}.
\end{equation}
Now, let  $V_{n+1}$ be a $\B_{n+1}-$measurable function. From (\ref{equ7}) and (\ref{equ9}) one gets
\begin{equation}
\label{equ10}
E(V_{n+1}|\B_n)(\pi^n)=\sum\limits_{j=1}^{k^n} V_{n+1}(\pi^n+\delta_j) \frac{(\pi^n(j)-\alpha)}{(\theta+n)}
+V_{n+1}((\pi^n,1))\frac{(\theta+\alpha k^n)}{(\theta+n)}.
\end{equation}

In Corollary $20$ in \cite{pitman1996} it was proven that the posterior distribution 
${\bf P}^{\alpha,\theta}(\cdot | \pi^n)$ chooses a random partition with the same law as
$(p_1,\dots,p_{k^n},p_{k^n+1}S')$, where: 
\begin{eqnarray}
\nonumber
&{}& (p_1,\dots,p_{k^n},p_{k^n+1}) \sim P_\D=\hbox{Dirichlet}(\pi^n(1)\!-\!\alpha, \dots, 
\pi^n(k^n)\!-\!\alpha,\theta\!+\!\alpha k^n), \\ 
\label{equ11}
&{}& S'\!=\!(S'(1),S'(2),\dots) \sim {\bf P}^{\alpha,\theta\!+\!\alpha k^n},
\end{eqnarray}
and they are independent. 

\medskip

Let  $P_\D^i$ be the marginal law of $P_\D$ in the $i-$th component. 
Then, $P^{k^n+1}_\D$ is Beta$(\theta+\alpha k^n, n-\alpha k^n)$. 
Since $G(S)=\sum_{i\in \NN}g(S(i))$, from (\ref{equ11}) one obtains,
\begin{equation}
\label{equ12}
E_{\bP(\cdot | \pi^n)}(G(S))=-\sum_{i=1}^{k^n} E_{P_\D^i}(g(p_i))
+E_{P_\D^{k^n+1}\otimes {\bf P}^{\alpha, \theta+\alpha k^n}}(G(p_{k^n+1} S')).
\end{equation}

For $G=\mH$, $g(x)=-x\log x$ and $\mH(p_{k^n+1} S')=-p_{k^n+1}\log p_{k^n+1}+H(S')$.
From $E_{\bP^{\alpha,\theta\!+\!\alpha k^n}}(\mH)=\psi(\theta+\alpha k^n+1)-\psi(1-\alpha)$
and $e_{(a,b)}(-X \log X)=\frac{a}{a+b}(\psi(a+b+1)-\psi(a+1))$ one can compute
the Bayesian posterior mean of the entropy, it is 
\begin{equation}
\label{equ13}
E_{\bP(\cdot | \pi^n)}(\mH) \!=\! \psi(\theta+\!n+\!1) 
\!-\!\frac{(\theta\!+\!\alpha k^n)}{\theta\!+\!n}\psi(1\!-\!\alpha)-\!\frac{1}{\theta\!+\!n}\!\sum_{i=1}^{k^n} 
(\pi^n(i)\!-\!\alpha)\psi(\pi^n(i)\!-\!\alpha\!+\!1).
\end{equation}
This relation was shown in  formulae $(10)$ in \cite{archer2012} and 
$(12)$ and $(15)$ in \cite{archer2014}.  

\medskip

To get the Bayesian posterior mean of the  Gini index it is convenient to introduce  
$\oGin(s)=1-\Gin(s)=\sum_{i\in \NN} s(i)^2$. 
From (\ref{equ12}) one gets
$$
E_{\bP(\cdot | \pi^n)}(\oGin)\!=\sum_{i=1}^{k^n} E_{P_\D^i}(p^2_i)+E_{P_\D^{k^n+1}}(p^2_{k^n+1})
E_{\bP^{\alpha,\theta\!+\!\alpha k^n}}(\oGin(S')).
$$
From $E_{\bP^{\alpha,\theta\!+\!\alpha k^n}}(\oGin)=\frac{1-\alpha}{\theta+1+\alpha k^n}$ 
and $e_{(a,b)}(X^2)=\frac{a(a+1)}{((a+b)(a+b+1)}$ we get
\begin{equation}
\label{equ14}
E_{\bP(\cdot | \pi^n)}(\Gin)
\!=1\!-\!\Big(\sum_{i=1}^{k^n} \frac{(\pi^n(i)\!-\alpha)(\pi^n(i)\!-\alpha\!+1)}{(\theta\!+n)(\theta\!+n\!+1)}\Big)
- \frac{(\theta\!+\alpha k^n)(1\!-\alpha)}{(\theta\!+n)(\theta\!+n\!+1)}.
\end{equation}
For $\kappa\in \NN$ one finds
\begin{equation}
\label{equ14A}
E_{\bP(\cdot | \pi^n)}(\Gin^{(\kappa)})
=1\!-\Big(\sum_{i=1}^{k^n} \frac{\prod_{r=0}^\kappa (\pi^n(i)\!-\alpha\!+r)}{\prod_{r=0}^\kappa (\theta\!+n\!+r)}\Big)
-\frac{(\theta\!+\alpha k^n)\prod_{r=0}^{\kappa-1}(1\!-\alpha\!+r)}{\prod_{r=0}^\kappa (\theta\!+n\!+r)}
\end{equation}

\smallskip

\subsection{Some computations on the martingale property}
\label{sec4.2}
In \cite{archer2014} it was computed the second moment of the entropy Bayesian estimator.
It was based upon the following  size-biased picking formula for sum-functions
$G(S)=\sum_{i\in \NN}g(S(i))$, see (11) in \cite{archer2014}:
$$
E_{\bP}\big(\sum_{i\neq j} g(S(i)g(S(j))\big)
=E_{\bP}\left(\frac{g(\hpi(1))g(\hpi(2))(1-\hpi(1))}{\hpi(1)\hpi(2)}\right).
$$
By using the independence of $W_1$, $W_2$, this gives 
$$
E_{\bP^{\alpha,\theta}}(G^2)=e_{(1-\alpha,\theta+\alpha)}\left(\frac{g(x)^2}{x}\right)
+ e_{(1-\alpha,\theta+\alpha)} \left(\frac{g(x)(1-x)}{x}\right) e_{(1-\alpha,\theta+2\alpha)}\left(\frac{g(x)}{x}\right).
$$             
Its computation serves to precise the Doob inequality 
$E_\bP\left(\sup\limits_{n\in\NN}\big(E_{\bP(\cdot | \pi^n)}(G)\big)^2\right)$.
The above formula together with $e_{(a,b)}((\log X)^2)=(\psi'(a)-\psi'(a+b))+(\psi(a)-\psi(a+b))^2$,
gives the second moment of the entropy supplied in  formula $(30)$ in \cite{archer2014},
\begin{eqnarray*}
E_{{\bf P}^{\alpha,\theta}}(\mH^2)&=&
\frac{1-\alpha}{\theta+1}\Big(\left(\psi'(2-\alpha)-\psi'(\theta+2))+(\psi(2-\alpha)-\psi(\theta+2)\right)^2\Big)\\
&{}&+\frac{\theta+\alpha}{\theta+1}(\psi(1-\alpha)-\psi(\theta+2))(\psi(1-\alpha)-\psi(\theta+1+\alpha)),
\end{eqnarray*}
For the Gini function $G=\Gin$ we use the value $e_{(a,b)}(X^3)$ to get,
$$
E_{{\bf P}^{\alpha,\theta}}(\Gin^2)
=\frac{\theta+2\alpha-1}{\theta+1}+\frac{(1-\alpha)(2-\alpha)(3-\alpha)}{(\theta+1)(\theta+2)(\theta+3)}+
\frac{(\theta+\alpha)(1-\alpha)^2}{(\theta+2)(\theta+1)(\theta+\alpha+1)}.
 $$

\medskip

Up to the end of this section, the martingale defined by the Bayesian posterior  mean of some integrable function 
$G$ is denoted by,
\begin{equation}
\label{equ15}
Z^{\rm G}_n(\Pi)=E_{\bP(\cdot| \pi^n)}(G),  \; \Pi\in \omN, n\in \NN. 
\end{equation}
We recall that $\pi^n$ is the $n-$th coordinate of $\Pi\in \omN$ and
$\pi^{n+1}\in \C(\pi^n)=\{\pi^n+\delta_{k^n}(j): j=1,..,k^n\}\cup \{(\pi^n,1)\}$.

\begin{proposition}
\label{prop2}
For the entropy $\mH$, the martingale difference 
$|Z^{\rm \mH}_{n+1}(\Pi)-Z^{\rm \mH}_n(\Pi)|$ is uniformly bounded,
\begin{equation}
\label{equ16}
\sup_{n\in \NN}\sup_{\Pi\in \omN}|Z^{\rm \mH}_{n+1}(\Pi)-Z^{\rm \mH}_n(\Pi)|<\infty.
\end{equation}
\end{proposition}

\begin{proof}
Let $A_n(\pi^n)=\sum_{i=1}^{k^n} (\pi^n(i)\!-\!\alpha)\psi(\pi^n(i)\!-\!\alpha\!+\!1)$.
From (\ref{equ13}) we get,
\begin{eqnarray*}
Z^{\rm \mH}_{n+1}(\pi^{n+1})\!-\!Z^{\rm \mH}_n(\pi^n)\!\!&=&\!\!(\psi(\theta\!+\!n\!+\!2)\!-\!\psi(\theta\!+\!n\!+\!1))
\!-\!\frac{(\theta+\alpha k^n)\psi(1-\alpha)}{(\theta+n)(\theta+n+1)}\\
&{}&\!\!-\frac{\alpha\psi(1-\alpha)}{(\theta+n+1)}{\bf 1}(k^{n+1}\!=\!k^n\!+\!1)
-\frac{A_n+(\theta+n)(A^1_n+A^2_n)}{(\theta+n)(\theta+n+1)},
\end{eqnarray*}
with 
\begin{eqnarray*}
A^1_n\!\!\!\!&=&\!\!\!\!\!\sum_{j=1}^{k^n}\!
\Big(\!(\pi^n(j)\!+\!1\!-\!\alpha)\psi(\pi^n(j)\!-\!\alpha\!+\!2)\! -\!(\pi^n(j)\!-\!\alpha)\psi(\pi^n(j)\!-\!\alpha\!+\!1)\! \Big) 
{\bf 1}(\pi^{n+1}\!=\!\!\pi^n\!+\!\delta_{k^n}(j)),\\
A^2_n\!\!\!\!&=&\!\!(1-\alpha)\psi(2-\alpha){\bf 1}(\pi^{n+1}\!=\!(\pi^n,1)).
\end{eqnarray*}
From the equality
\begin{equation}
\label{equ17}
\forall x>0:\quad x\psi(x+1)-(x-1)\psi(x)=\psi(x)+1\,, 
\end{equation}
we get $(\pi^n(j)\!+\!1\!-\!\alpha)\psi(\pi^n(j)\!-\!\alpha\!+\!2)
\!-\!(\pi^n(j)\!-\!\alpha)\psi(\pi^n(j)\!-\!\alpha\!+\!1)\!=\!\psi(\pi^n(i)\!-\!\alpha\!+\!1)\!+\!1$, so
$$
A^1_n\!=\!\sum_{j=1}^{k^n}(\psi(\pi^n(j)\!-\!\alpha\!+\!1)+\!1) {\bf 1}(\pi^{n+1}\!=\!\pi^n\!+\!\delta_{k^n}(j)).
$$
Since $\{\pi^{n+1}\!=\!(\pi^n,1)\!\}\!=\!\{k^{n+1}\!=\!k^n\!+\!1\}$ and by using the equality
$(1-\alpha)\psi(2\!-\alpha)\!+\alpha\psi(1-\alpha)\!=\!\psi(1\!-\alpha)\!+1$, which also follows from (\ref{equ17}), we get 
\begin{eqnarray}
\nonumber
Z^{\rm \mH}_{n+1}(\pi^{n+1}\!)\!-\!Z^{\rm \mH}_n(\pi^n)
\!\!\!\!&=&\!\!\!\!
(\psi(\theta\!+\!n\!+\!2)\!-\!\psi(\theta\!+\!n\!+\!1))\!+
\!\frac{(\theta\!+\!\alpha k^n)(1\!-\!\alpha)\psi(1\!-\!\alpha)\!-\!A_n(\pi^n)}{(\theta\!+\!n)(\theta\!+\!n\!+\!1)}\\
\nonumber
&{}& \,
-\sum_{j=1}^{k^n}\! \frac{(\psi(\pi^n(j)\!-\!\alpha\!+\!1)\!+\!1)}
{(\theta\!+\!n\!+\!1)} {\bf 1}(\pi^{n+1}\!=\pi^n\!+\!\delta_{k^n}(j))\\
\label{equ18}
&{}& \; -\frac{\psi(1\!-\alpha)\!+1}{(\theta\!+n\!+1)}{\bf 1}(\pi^{n+1}\!=(\pi^n,1)).
\end{eqnarray}
To get uniform bounds of these terms, we use $|\psi(x)- \log(x)|\le x^{-1}$ as $x\to \infty$. Then 
$\psi(\theta\!+\!n\!+\!2)\!-\!\psi(\theta\!+\!n\!+\!1)\to 0$ as $n\to \infty$. Since $\psi$ is increasing, we also get
$$
\frac{(\psi(\pi^n(j)\!-\!\alpha\!+\!1)+1)}{(\theta+n+1)}\le \frac{\psi(n-\alpha+1)}{(\theta+n+1)}\to 0 \hbox{ as } n\to \infty.
$$
Then, the last two terms at the right hand side in (\ref{equ18}), which are on the disjoint elements 
$\pi^{n+1}\in \C(\pi ^n)$, have a uniform bound on $n$ and $\Pi$. Finally, 
$A_n(\pi^n)\le \psi(n\!-\alpha\!+1)(\sum_{i=1}^{k^n} \pi^n(i))\!=\!n \psi(n\!-\alpha\!+1)$ and then 
$\frac{A_n}{(\theta+n)(\theta+n+1)}$ 
is uniformly bounded on  $n$ and $\Pi$.  Then, the
martingale difference $Z^{\rm \mH}_{n+1}(\pi^{n+1})-Z^{\rm \mH}_n(\pi^n)$ is uniformly bounded, 
and (\ref{equ16}) is satisfied.
\end{proof}

\medskip

\begin{remark}
\label{rem3}
Even if $Z^{\rm \mH}_{n+1}(\pi^{n+1})-Z^{\rm \mH}_n(\pi^n)$ is uniformly bounded 
one cannot state a  Central Limit Theorem for the martingale difference as in \cite{ibra} 
(also see the Introduction in \cite{ouchti}), 
because besides the uniformly boundedness this property also requires 
$\sum_{n\ge 0}\sigma_{\rm \mH}^2(n)=\infty$, 
where 
$$\sigma_{\rm \mH}^2(n)=E_{\bP}(E((Z^{\rm \mH}_{n+1}-Z^{\rm \mH}_n)^2|\B_n))
=E_{\bP}(E((Z^{\rm \mH}_{n+1})^2|\B_n)-(Z^{\rm \mH}_n)^2).
$$
But, this property does not hold, on the contrary we have 
$\sum_{n\ge 0}\sigma_{\rm \mH}^2(n) \le E_\bP(\mH^2)<\infty$.
For $G\!=\!\Gin$,  
$|Z^{\rm \Gin}_{n+1}(\Pi)\!-\!Z^{\rm \Gin}_n(\Pi)|\!\le \!1$ and
$\sum_{n\ge 0}\sigma_{\rm \Gin}^2(n) \le E_\bP(\Gin^2)<\infty$.
\end{remark}

\medskip

\section{Bayesian posterior and plug-in estimators for the entropy and the Gini indexes}
\label{sec5}
We will state some properties shared by the Bayesian and the plug-in estimators, 
for the entropy and the (generalized) Gini indexes. 
To a function $G$ one associates two sequences of estimators, 
$$
(E_{\bP(\cdot|\pi^n)}(G): n\in \NN) \hbox{ and } (G(\pi^n): n\in \NN), 
$$
the second one is called the plug-in estimator. We seek to state some relations between them. 

\smallskip

In next Section we describe their limit behavior for exchangeable random partitions. 
For the Bayesian posterior estimator this behavior is given by Proposition \ref{prop1} and for the 
plug-in estimators the behavior is obtained from \cite{antos}. 

\smallskip
 
Further, in Section (\ref{sec5.2}) and ({\ref{sec5.3}), we look for relations between these two estimators 
for their one step evolution in time and maxmin properties at some fixed time. This is discussed 
for the entropy and the generalized Gini indexes showing that in the framework of PDP,  they have 
similar properties, in
particular at the times of discovery of new species. In these sections the generalized Gini index is assumed
to have an integer coefficient.

\medskip

\subsection{Pointwise limit of the plug-in estimators for exchangeable partitions}
\label{sec5.1} 
We recall that the sequence $\X=(X_l: l\ge 1)$ of i.i.d. Uniform r.v's in $[0,1]$ is independent 
of $S$. To be in the body of \cite{antos} one fixes $s\in \aS$, consider $\J^s$ its associated fixed 
countable partition of the unit interval and define $(Y_l: l\in \NN)$ by
$Y_l=i$ if $X_l\in J^s(i)$ for $i\in \NN$. The sequence $(Y_l: l\in \NN)$ is i.i.d. respect to 
$\PP(\cdot |s)$, and this is the frame used in \cite{antos} to
study the plug-in estimators of functions of the type $G(s)=\sum_{i\in \NN} g_i(s(i))$, that is the partition
of the unit interval is fixed. 

\medskip

In our original setting we assume $E_P(|G|)<\infty$, 
so $G(s)$ is finite a.s. The plug-in estimator of $G$ at step $n$ is $G(\pi^n/n)$. 
Since $G$ is symmetric, \cite{antos} serves to study the plug-in estimators of sum-type functions 
$G(s)=\sum_{i\in \NN} g(s(i))$. So, if we are able to use \cite{antos} and show that 
$G(\pi^n/n)\to G\;\; \PP(\cdot |s)$-a.s. and this holds $P$-a.s in $\aS$, then
one should get $G(\pi^n/n)\to G\;\; \PP-$a.s. Thus, we get:

\smallskip

\begin{proposition}
\label{prop3}
In the frame of exchangeable random partitions and assuming $E_P(\mH)<\infty$ (for instance in the PDP case) 
we have
\begin{eqnarray*}
&{}&\lim\limits_{n\to \infty} E_{\bP(\cdot | \pi^n)}(\Gin^{(\kappa)})= \Gin^{(\kappa)}
=\lim\limits_{n\to \infty} \Gin^{(\kappa)}(\pi^n/n) \; \; \PP\!-\!\hbox{a.s. and in } L^p(\PP) \; \forall p\ge 1;\\
&{}& \hbox{ and } 
\lim\limits_{n\to \infty} E_{\bP(\cdot | \pi^n)}(\mH)= \mH
=\lim\limits_{n\to \infty}  \mH(\pi^n/n) \; \; \PP\!-\!\hbox{a.s.},
\end{eqnarray*}
where the left limit of the last relation also holds in $L^1(\PP)$. 
\end{proposition}

\begin{proof}
We use   (\ref{equ3}) in Proposition \ref{prop1} for the limit of the Bayesian estimators and for
the plug-in estimators this is obtained straightforwardly:

$\cdot\;$ From \cite{antos} pp. 169-172, we get  
$\Gin^{(\kappa)}=\lim\limits_{n\to \infty}  \Gin^{(\kappa)}(\pi^n/n)$
$\PP(\cdot |s)$-a.s., and this holds $P$-a.s in $\aS$, so the $\PP-$ a.s. convergence follows.
Since it is uniformly bounded by $1$ the convergence holds for all $L^p(\PP)$, $p\ge 1$;

\smallskip

$\cdot\;$ From  Theorem $2$ in \cite{antos} we get
$\mH=\lim\limits_{n\to \infty}  \mH(\pi^n/n)$ $\PP(\cdot |s)$-a.s., since this holds $P$-a.s in $\aS$
the $\PP-$a.s. convergence  follows.
\end{proof}

\medskip

We notice that the tools and results supplied in \cite{antos} are not sufficient for showing the $L^p(\bP)$ 
convergence of  the plug-in estimators that are not uniformly bounded in $s$, because 
besides the convergence of the estimator on $L^p(\bP(\cdot|s))$ one requires to have a control 
of its behavior on $s\in \aS$.  
For instance,  in Theorem $2$ in \cite{antos} it is shown the $L^2(\PP(\cdot |s))$ convergence 
of the entropy, and if one wishes to use this result to state the $L^2(\PP)-$ convergence, one
should have an appropriated control on $s\in \aS$ of a sequence depending on 
$E_{\bP(\cdot |s)}(\mH(\pi^n/n))$, which, to our view, is not the case or at least cannot be stated
in an easy way for general exchangeable partitions. 

\medskip

In relation to above result, in \cite{archer2014}  it was shown that in the framework of the PDP one has 
$|\mH(\pi^n/n)-E_{\bP(\cdot |\pi^n)}(\mH)|\to 0$  in probability as $n\to \infty$. The  
proof uses heavily formula (\ref{equ13}) and
Proposition $2$ in \cite{gnedin}. 

\medskip

\subsection{Maxmin property along trajectories of the plug-in estimators}
\label{sec5.2}
Let us fix some $s\in \aS$ and  $\J^s$ be a fixed countable partition of the unit interval. 
Let us describe what happens along a trajectory of the plug-in estimator. So, the observations 
$X_1,\dots,X_n$ are grouped into classes enumerated sequentially following the time of discovery, 
this defines a partition $\xi^n$ of $I_n=\{1,\dots,n\}$ and its classes are also called 
the species  discovered up to $n$. We set $\pi^n=(\pi^n(j)=\# \xi(j): j=1,\dots,k^n)$. 
Let $j^*$ be the class containing $X_n$, 
so, if $j ^*$ is firstly observed before or at time $n-1$ then
$k^n=k^{n-1}$ and $\pi^{n}(j^*)=\pi^{n-1}(j^*)+1$, and if $X_n$ defines a new class then
$k^n=k^{n-1}+1$, $j^*=k^n$ and $\pi^n(k^n)=1$.

\smallskip

We will describe four properties seen in a trajectory of the plug-in estimators for the entropy and 
the Gini functions, that will be retrieved for the Bayesian posterior estimators.  The results hold for 
the generalized Gini indexes, but with parameter $\kappa\in \NN$, and this condition is always assumed.

\medskip

We put  $0\log 0=0$. Let us consider the following sequences for $m\in \NN$, 
$$
\gamma_0(m)\!=\!m\log m-(m-1)\log(m-1), \; \gamma_\kappa(m)\!=\!m^{\kappa+1}-(m-1)^{\kappa+1}
=\sum_{l=0}^\kappa \binom{\kappa}{l}(m-1)^l\,,
$$
that are strictly increasing on $m\in \NN$ . Also $\gamma_0(m)\ge 0$, $\gamma_\kappa(m)\ge 1$. 

\medskip

Recall that $\mN_n$ is the set of all elements $\pi^n$. It is useful to denote $\1^n=(1,..,1)$ 
the $\pi^n$ vector of length $n$ constituted only of $1's$, so
when each variable $X_l$ defines a different class for $=1,..,n$, or equivalently when $k^n=n$. 

\smallskip

The plug-in entropy and the Gini functions at step $n$ are respectively given by 
\begin{eqnarray*}
\mH(\pi^n) &=& - \sum_{j=1}^{k^n}
\frac{\pi^n(j)}{n} \log \left(\frac{\pi^n(j)}{n}\right)=-\frac{1}{n}\big(\sum_{j=1}^{k^n}\pi^n(j)\log \pi^n(j)-n\log n\big),\\
\Gin^{(\kappa)}(\pi^n) &=&1- \sum_{j=1}^{k^n}\left(\frac{\pi^n(j)}{n}\right)^{\kappa+1}
=-\frac{1}{n^{\kappa+1}}\big(\sum_{j=1}^{k^n}(\pi^n(j))^{\kappa+1}-n^{\kappa+1}\big).
\end{eqnarray*}
Let $I_n$ be endowed with the uniform probability that gives a weight $1/n$ to each $j\in I_n$. 
Hence $\mH(\pi^n)$ is the Shannon entropy of the partition $\xi^n$ of $I_n$.

\medskip

${\bf Ia}$. One has $\mH(\1^n)=\log n$ and $\Gin^{(\kappa)}(\1^n)=(n^{\kappa}-1)/n^\kappa$. Moreover, 
$0\le \mH(\pi^n)\le \mH(\1^n)$ and $0\le \Gin(\pi^n)\le \Gin(\1^n)$, so
for both, $\mH(\pi^n)$ and $\Gin(\pi^n)$, the maximum value is attained at $\1^n$,
and the minimum value is $0$ and happens only when $k^n=1$.

\medskip

${\bf IIa}$. We will introduce some quantities whose meaning for the entropy is the following one.
Let us select uniformly an individual in $I_n$. 
The entropy $\mH(\pi^n)$ is the mean information of $\xi^n$ defined by the species of 
the individuals up to $n$, 
and $\mH(\1^n)$ corresponds to the mean information of the individuals. 
Then, $\mH(\1^n)-\mH(\pi^n)$ is the mean information needed to identify an individual given 
that its species is determined, or the conditional entropy of the individual given the species.

\smallskip

Now, let us makes $n$ independent experiences, each one of them selecting uniformly an individual in 
$I_n$. Then, $n \mH(\pi^n)$ is the mean information given by the species of the $n$
independent individuals  and $n(H(\1^n)-\mH(\pi^n))$ is the mean information of the $n$ 
independent individuals given that their species are determined.

\smallskip

For the Gini index of parameter $k\in \NN$: $\Gin^{(\kappa)}(\1^n)\!-\!\Gin^{(\kappa)}(\pi^n)$, is
the probability that the species of $\kappa+1$ independent individuals are the same, but the individuals 
are not all equal. 

\smallskip

Define the sequences, 
\begin{equation}
\label{equ19}
\ell^{\rm \mH}(\pi^n)\!=\!n(\mH(\1^n)-\mH(\pi^n)),\; 
\ell^{\rm {\Gin^{(\kappa)}}}(\pi^n)\!=\!n^{\kappa+1}(\Gin^{(\kappa)}(\1^n)\!-\!\Gin^{(\kappa)}(\pi^n)), \; n\ge 1, 
\end{equation}
and take $\ell^{\rm \mH}(\pi^0)=0=\ell^{\rm {\Gin^{(\kappa)}}}(\pi^0)$. Both quantities, $\ell^{\rm \mH}(\pi^n)$
and $\ell^{\rm {\Gin^{(\kappa)}}}(\pi^n)$ are nonnegative.

\medskip

We have $\ell^{\rm \mH}(\pi^n)=\sum_{j=1}^{k^n}\pi^n(j)\log \pi^n(j)$, so
$\ell^{\rm \mH}(\pi^n)-\ell^{\rm \mH}_{n-1}(\pi^{n-1})=\gamma_0(\pi^n(j^*))\geq 0$,
and $\ell^{\rm \mH}(\pi^n)-\ell^{\rm \mH}(\pi^{n-1})=0$ only when $\pi^n(j^*)=1$, so if a 
new class is observed at time $n$.

\smallskip

On the other hand
$\ell^{\rm {\Gin^{(\kappa)}}}(\pi^n)=-n+\sum_{j=1}^{k^n}\pi^n(j)^{\kappa+1}$, then
$\ell^{\rm {\Gin^{(\kappa)}}}(\pi^n)-\ell^{\rm {\Gin^{(\kappa)}}}(\pi^{n-1})
=\big(\pi^n(j^*)^{\kappa+1}-(\pi^n(j^*)-1)^{\kappa+1}\big)-1\ge 0$,
that vanishes only when $\pi^n(j^*)=1$, similarly as it happens for the entropy.

\medskip

${\bf IIIa}$. Consider the sequences  $n \mH(\pi^n)$ and $n^{\kappa+1} \Gin(\pi^n)$. We have 
\begin{eqnarray*}
\Delta_n^{\rm \mH}(\pi^n)&=&n \mH(\pi^n)-(n-1)\mH(\pi^{n-1})
=\gamma_0(n)-\gamma_0(\pi^n(j^*)) \hbox{ and }\\
\Delta_n^{\rm {\Gin}^{(\kappa)}}(\pi^n)&=&n^{\kappa+1} \Gin^{(\kappa)}(\pi^n)-(n-1)^{\kappa+1} \Gin^{(\kappa)}(\pi^{n-1})
=\gamma_\kappa(n)-\gamma_\kappa(\pi^n(j^*)).
\end{eqnarray*}
Then,  $\Delta_n^{\rm \mH}(\pi^n)$ and $\Delta_n^{\rm {{\Gin}^{(\kappa)}}}(\pi^n)$ 
are nonnegative and they vanish for some $n$ 
only when $\pi^n(j^*)=n$, that is when all individuals belong to a single class, and they attain a maximum value 
when $\pi^n(j^*)=1$, so if a new class is observed at $n$.

\medskip

${\bf IVa}$. Let $\mN_n(k)$ be the set of all $\pi^n\in \mN_n$ with $k$ classes. Let $j,l$ be two 
different coordinates in $\{1,...,k\}$.
Let $\pi^n\in \mN_n(k)$ be such that $\pi^n(j)\ge \pi^n(l)+2$. We define $\opi^n \in \mN_n(k)$  
by $\opi^n(i)=\pi^n(i)$ for $i\in \{1,..,k\}\setminus \{j,l\}$ and 
$\opi^n(j)=\pi^n(j)-1$, $\opi^n(l)=\pi^n(l)+1$. We have the property,
\begin{equation}
\label{equ20}
\mH(\opi^n)\ge \mH(\pi^n)\hbox{ and } \Gin^{(\kappa)}(\opi^n)\ge \Gin^{(\kappa)}(\pi^n).
\end{equation}
This follows from  
$\,n(\mH(\opi^n)-\mH(\pi^n))=\gamma_0(\pi^n(j))-\gamma_0(\pi^n(l)+1)\ge 0$ and
$n^{\kappa+1}(\Gin^{(\kappa)}(\opi^n)-\Gin^{(\kappa)}(\pi^n))=\gamma_\kappa(\pi^n(j))-\gamma_\kappa(\pi^n(l)+1)\ge 0$.
 For the inequalities we use that $\gamma_0$ and $\gamma_\kappa$ are increasing functions.

\medskip

By iterating the inequality (\ref{equ20}) one shows that:
\begin{eqnarray*}
&{}&\mH(\pi_{\rm min(k)}^n)=\min\limits_{\pi^n\in \mN_n(k)}\mH(\pi^n) \,,  \;\,
\mH(\pi_{\rm max(k)}^n)=\max\limits_{\pi^n\in \mN_n(k)}\mH(\pi^n); \\
&{}& \Gin^{(\kappa)}(\pi_{\rm min(k)}^n)=\min\limits_{\pi^n\in \mN_n(k)}\Gin^{(\kappa)}(\pi^n)\,, \;\,
\Gin^{(\kappa)}(\pi_{\rm max(k)}^n)=\max\limits_{\pi^n\in \mN_n(k)}\Gin^{(\kappa)}(\pi^n);
\end{eqnarray*}
where $\pi^n_{\rm min(k)}$ is constituted by one class containing $n-k+1$ individuals and the other 
$k-1$ classes are singletons and $\pi^n_{\rm max(k)}$ is constituted by: $s$ classes with 
$\lceil n/k \rceil$ individuals and $k-s$ with  $\lfloor n/k \rfloor +1$ individuals, being $s=k-(n-k\lfloor n/k \rfloor)$.
Here, as usual, by $\lfloor n/k \rfloor$ we mean the integer part of $n/k$. 
We will say that $\pi^n_{\rm max(k)}$ divides $\{1,..,n\}$ into $k$ parts 'as uniform as possible'. 

\medskip

\subsection{Maxmin properties of the Bayesian posterior means for the PDP}
\label{sec5.3}
We will define the analogous notions and state similar results as those shown in ${\bf {Ia, IIa, IIIa, IVa}}$ 
in Section (\ref{sec5.2}) for the Bayesian posterior means of the entropy and the Gini indexes.
But, this is be proven in the framework of the PDP$(\alpha,\theta)$ with 
$0\leq \alpha< 1$ and $\theta>-\alpha$. These results
are stated for the entropy and the generalized Gini index with $\kappa\in \NN$ but the proofs will
be made in detail for the Gini index $\kappa=1$, to avoid heavy notation. 
But at some parts of the proofs we will put the formulae for the generalized Gini index. 

\medskip

We recall that the Bayesian estimators for $\mH$, $\Gin$ and $\Gin^{(\kappa)}$ 
are given in (\ref{equ13}), (\ref{equ14}) and (\ref{equ14A}), respectively.  
In many of the proofs  make in these sections we will use the following equalities for $x>0$: 

\smallskip

-for the entropy $x\psi(x+1)-(x-1)\psi(x)=\psi(x)+1$ given in (\ref{equ17}) and, 

\smallskip

-for the Gini index of parameter $\kappa\!\in \!\NN$,    
$\prod\limits_{r=0}^\kappa (x\!+r)-\prod\limits_{r=0}^{\kappa} (x\!-1\!+r)
=(\kappa\!+\!1)\prod\limits_{r=0}^{\kappa-1}(x\!+r)$.

\smallskip

These functions, $\psi(x)+1$ and $(\kappa+1)\prod_{r=0}^{\kappa-1}(x+r)$ 
are strictly increasing in $(0,\infty)$.  For the Gini index with $\kappa=1$, 
the product equality reduces to $x(x+1)-(x-1)x=2x$.

\medskip

${\bf Ib}$. The similar statement to ${\bf Ia}$ for the argmin of  the Bayesian posterior means is:
\begin{lemma}
\label{lemma1}
The minimum of the Bayesian posterior means of the entropy and the Gini function in $\mN_n$,
 is attained for a single class $k^n=1$.
\end{lemma}
This result will be shown further, at the end of ${\bf IVb}$.

\medskip

Let us now prove the similar statement to the argmax of the Bayesian posterior means. 

\smallskip

\begin{lemma}
\label{lemma2}
The maximum of the Bayesian posterior means of the entropy and the Gini function in $\mN_n$,
 is attained at $\1^n$, that is 
\begin{equation}
\label{equ21}
\max\limits_{\pi^n\in \mN_n} \! E_{\bP(\cdot | \pi^n)}(\mH)\!=\!E_{\bP(\cdot | \1^n)}(\mH) \hbox{ and }
\max\limits_{\pi^n\in \mN_n} \! E_{\bP(\cdot | \pi^n)}(\Gin^{(\kappa)})\!=\!E_{\bP(\cdot | \1^n)}(\Gin^{(\kappa)}).
\end{equation} 
\end{lemma}

\begin{proof}
Let us fix $\pi^{n-1}$ with $k=k^{n-1}$ classes. Assume $\hpi^n$ also has $k$ classes 
and for a unique  $j$, $\hpi^n(j)=\pi^{n-1}(j)+1$, for all the other coordinates $i\neq j$, 
$\hpi^n(i)=\pi^{n-1}(i)$. Also consider $\pi^n_+=(\pi^{n-1},1)$ that has $k+1$ classes, 
the last one having a unique element.  By definition,
$\pi^n_+(j)=\hpi^n(j)-1$ and for all other $i\neq j$ in $\{1,..,k\}$ one has $\hpi^n_+(i)=\hpi^n(i)$.  
Let us show that for $\mH$ and $\Gin$ it holds:
\begin{eqnarray}
\nonumber
&{}&(\theta+n) \big(E_{\bP(\cdot | \pi^n_+)}(\mH)-E_{\bP(\cdot | \hpi^n)}(\mH)\big)>0 \hbox{ and }\\
&{}&
\label{equ22}
\big(\prod_{r=0}^\kappa(\theta+n\!+r)\big) 
\big(E_{\bP(\cdot | \pi^n_+)}(\Gin^{(\kappa)})-E_{\bP(\cdot | \hpi^n)}(\Gin^{(\kappa)})\big)>0.
\end{eqnarray}
Let us prove it for the entropy. From (\ref{equ13}) and by using (\ref{equ17}) one obtains
\begin{eqnarray*}
&{}&(\theta+n) (E_{\bP(\cdot | \pi^n_+)}(\mH)-E_{\bP(\cdot | \hpi^n)}(\mH))\\
&{}&=-\alpha\psi(1-\alpha)-(1-\alpha)\psi(2-\alpha)+1+\psi(\pi^{n-1}(j)+1-\alpha)\\
&{}& =-\psi(1-\alpha)+\psi(\pi^{n-1}(j)+1-\alpha)>0,
\end{eqnarray*}
because $\pi^{n-1}(j)+1>1$. 
Now let us show it for the Gini function with $\kappa=1$. From  (\ref{equ14}) and by using $(x+1)x-x(x-1)=2x$ 
for $x=\hpi^n(j)=\pi^{n-1}(j)+1$, one gets
\begin{eqnarray*}
&{}&(\theta+n) (\theta+n+1)(E_{\bP(\cdot | \pi^n_+)}(\Gin)-E_{\bP(\cdot | \hpi^n)}(\Gin))\\
&{}& =-\alpha(1-\alpha)-(1-\alpha)(2-\alpha)+2(\pi^{n-1}(j)+1-\alpha)\\
&{}&=-2(1-\alpha)+2(\pi^{n-1}(j)+1-\alpha)=2\pi^{n-1}(j)> 0.
\end{eqnarray*}
For the generalized Gini index with $\kappa\in \NN$ it follows from,
$$
\big(\!\prod_{r=0}^\kappa(\theta+n\!+r)\!\big)\!\big(\!E_{\bP(\cdot | \pi^n_+)}(\Gin^{(\kappa)})-E_{\bP(\cdot | \hpi^n)}(\Gin^{(\kappa})\big)
\!=\!(\kappa+\!1)\big(\!\prod_{r=1}^\kappa(\pi^{n-1}(j)\!+r\!-\alpha)-\prod_{r=1}^\kappa\!(r\!-\alpha)\!\big)\!>\!0.
$$
The iteration of equalities (\ref{equ22}) allow to get relation (\ref{equ21}).
\end{proof}

\medskip

${\bf IIb}$. Define the following sequences, 
\begin{eqnarray}
\label{equ23}
&{}&\aL^{\rm \mH}(\pi^n)
=(\theta+n)\left(E_{\bP(\cdot |1^n)}(\mH)-E_{\bP(\cdot | \pi^n)}(\mH)\right) \hbox{ and }\\
\nonumber
&{}& \aL^{\rm {\Gin^{(\kappa)}}}(\pi^n)
=\big(\prod_{r=0}^\kappa(\theta+n+r)\big) \left(E_{\bP(\cdot |1^n)}(\Gin^{(\kappa)})-E_{\bP(\cdot | \pi^n)}(\Gin^{(\kappa)})\right),
\end{eqnarray}
which are adapted to the filtration $(\sigma(\Pi_n): n\in \NN)$ and from Lemma \ref{lemma2} 
they are nonnegative. They satisfy the following properties 
(the one devoted to the entropy was firstly proved in \cite{martinez} Theorem $4.4$):

\begin{proposition}
\label{prop4}
 The functions $\aL^{\rm \mH}(\pi^n)-\aL^{\rm \mH}(\pi^{n-1})$ and 
$\aL^{\rm{\Gin^{(\kappa)}}}(\pi^n)-\aL^{\rm {\Gin^{(\kappa)}}}(\pi^{n-1})$ are nonnegative and they
vanish only when $\pi^n=(\pi^{n-1},1)$, so if a new class appears at $n$.
\end{proposition}

\begin{proof}
We make it for $\mH$ and $\Gin$. We have, 
\begin{eqnarray*}
&{}& (\theta+n)E_{\bf P(\cdot | \1^n)}(\mH)\!=\!(\theta\!+n)\psi(\theta\!+n\!+1)
\!-(\theta\!+\alpha n)\psi(1\!-\alpha)\!-n(1\!-\!\alpha)\psi(2\!-\alpha);\\
&{}& (\theta\!+\!n) (\theta\!+n\!+\!1)E_{\bP(\cdot | \1^n)}(\Gin)
\!=(\theta\!+n) (\theta\!+n\!+\!1)\!-n(1\!-\alpha)(2-\!\alpha)\!-(\theta\!+\alpha n)(1\!-\!\alpha).
\end{eqnarray*}

We continue with the notation used in the proof of Lemma \ref{lemma2}. 
So, $\pi^n_+=(\pi^{n-1},1)$ and $\hpi^n$ has $k=k^{n-1}$ classes 
and for a unique  $j$, $\hpi^n(j)=\pi^{n-1}(j)+1$, for all other $i\neq j$, 
$\hpi^n(i)=\pi^{n-1}(i)$. 

\medskip

For the entropy we have
$$
\aL^{\rm \mH}(\hpi^n)
=-(\theta+\alpha (n\!-\!k))\psi(1\!-\!\alpha)\!-\!n(1\!-\!\alpha)\psi(2\!-\!\alpha)
+\sum_{i=1}^{k}  (\hpi^n(i)\!-\!\alpha)\psi(\hpi^n(i)\!-\!\alpha+1).
$$
Hence
\begin{eqnarray*}
&{}&\aL^{\rm \mH}(\hpi^n)-\aL^{\rm \mH}(\pi^{n-1})\\
&{}&=\alpha\psi(1\!-\!\alpha)-(1\!-\!\alpha)\psi(2\!-\alpha)+1+\psi(\hpi^{n}(j)\!-\alpha)
=-\!\psi(1-\!\alpha)\!+\!\psi(\hpi^n(j)-\alpha)> 0,
\end{eqnarray*}
because $\hpi^n(j)=\pi^{n-1}(j)+1>1$. On the other hand
\begin{eqnarray*}
\aL^{\rm \mH}(\pi^n_+)
&=&-(\theta+\alpha (n\!-\!k-1))\psi(1\!-\!\alpha)-n(1\!-\alpha)\psi(2\!-\!\alpha)\\
&{}& +\sum_{i=1}^{k}   (\pi^{n-1}(i)\!-\!\alpha)\psi(\pi^{n-1}(i)\!-\alpha\!+1)+(1\!-\alpha)\psi(2\!-\alpha),
\end{eqnarray*}
then,
$$
\aL^{\rm \mH}(\pi^n_+)-\aL^{\rm \mH}(\pi^{n-1})
=-(1-\alpha)\psi(2-\alpha)+(1-\alpha)\psi(2-\alpha)=0.
$$
For the Gini function we have 
$$
\aL^{\rm \Gin}(\hpi^n)
=-(\theta+\alpha (n\!-\!k))(1\!-\!\alpha)\!-\!n(1\!-\!\alpha)(2\!-\!\alpha)
+\sum_{i=1}^{k} (\hpi^n(i)\!-\!\alpha)(\hpi^n(i)\!-\!\alpha+1).
$$
Hence
$$
\aL^{\rm \Gin}(\hpi^n)-\aL^{\rm \Gin}(\pi^{n-1})
=-\alpha(1\!-\!\alpha)\!-(1\!-\!\alpha)(2\!-\alpha)\!+2(\hpi^{n}(j)\!-\alpha)
=2(\hpi^n(j)\!-1)\!> 0.
$$
On the other hand
$$
\aL^{\rm \Gin}(\pi^n_+)
\!=\!-(\theta+\alpha (n\!-\!k\!-1))(1\!-\!\alpha)\!-\!n(1\!-\!\alpha)(2\!-\!\alpha)
\!+\sum_{i=1}^{k} (\pi^{n-1}(i)\!-\!\alpha)(\pi^{n-1}(i)\!-\!\alpha+1)+(1\!-\alpha)(2\!-\alpha),
$$
then,
$$
\aL^{\rm \mH}(\pi^n_+)-\aL^{\rm \mH}(\pi^{n-1})=-(1-\alpha)(2-\alpha)+(1-\alpha)(2-\alpha)=0.
$$
The result is shown for the Gini index. For the generalized Gini index with $\kappa\in \NN$ we have,
$$
\aL^{\rm {\Gin^{(\kappa)}}}(\hpi^n)-\aL^{\rm  {\Gin^{(\kappa)}}}(\pi^{n-1})
=-\prod_{r=0}^\kappa(r\!-\!\alpha)- \prod_{r=0}^\kappa(r\!+1\!-\!\alpha)
+(\kappa+1)\prod_{r=0}^{\kappa-1}(\hpi^n(j)\!-\!\alpha+r),
$$
which is $>0$. Finally, one easily check that
$\aL^{\rm {\Gin^{(\kappa)}}}(\pi^n_+)-\aL^{\rm {\Gin^{(\kappa)}}}(\pi^{n-1})=0$.
\end{proof}

\medskip

The properties stated in ${\bf IIa}$ and Proposition \ref{prop4},  for the sequences (\ref{equ19}) 
and (\ref{equ23}) respectively, can be summarized in: 

\begin{corollary}
\label{cor1}
The following five properties are equivalent (where $\kappa\in \NN$):
\begin{eqnarray*}
&{}& (i)\; \ell^{\rm {\Gin^{(\kappa)}}}(\pi^n)=\ell^{\rm {\Gin^{(\kappa)}}}(\pi^{n-1}), \quad   (ii) \; \ell^{\rm \mH}(\pi^n)=\ell^{\rm \mH}(\pi^{n-1}),\\
&{}& (iii)\; \pi^n=(\pi^{n-1},1) \, \hbox{ that is a new species is discovered at } n,\\
&{}& (iv)\;  \aL^{\rm {\Gin^{(\kappa)}}}(\pi^n)=\aL^{\rm {\Gin^{(\kappa)}}}(\pi^{n-1}), \quad (v)\; \aL^{\rm \mH}(\pi^n)=\aL^{\rm \mH}(\pi^{n-1}).
\end{eqnarray*}
The first three are equivalent for every exchangeable random partition and the last three are equivalent in the PDP frame. 
\end{corollary}

\medskip

${\bf IIIb}$. Now we prove: 
 
\begin{proposition}
\label{prop5}
 The functions $\Delta^{\rm \mH}_n(\pi^n)=(n+\theta)E_{\bf P(\cdot | \pi^n)}(\mH)-(n-1+\theta)E_{\bf P(\cdot | \pi^{n-1})}(\mH)$ and 
$\Delta^{\rm {\Gin^{(\kappa)}}}(\pi^n)=\big(\prod_{r=0}^\kappa(\theta+n\!+r)\big) E_{\bf P(\cdot | \pi^n)}(\Gin^{(\kappa)})-
\big(\prod_{r=0}^\kappa(\theta+n\!-1\!+r)\big) E_{\bf P(\cdot | \pi^{n-1})}(\Gin^{(\kappa)})$
are both nonnegative and attain their maxima at $\pi^n=(\pi^{n-1},1)$.
\end{proposition}

\begin{proof}
Let $k=k^{n-1}$ and firstly assume $\pi^n=\hpi^n$ so with $k$ classes, 
$\hpi^n(j)=\pi^{n-1}(j)+1$ and for all other $i\neq j$, $\hpi^n(i)=\pi^{n-1}(i)$.  
From (\ref{equ13}) and (\ref{equ17}) we obtain
\begin{eqnarray*}
\Delta^{\rm \mH}(\hpi^n)&=&(\theta+n)\psi(\theta+n+1)-(\theta+n-1)\psi(\theta+n)\\
&{}&-(\pi^{n-1}(j)\!-\alpha\!+1)\psi(\pi^{n-1}(j)\!+\alpha\!+2)\!-(\pi^{n-1}(j)\!-\alpha)\psi(\pi^{n-1}(j)\!-\alpha\!+1)\\
&=& \psi(\theta\!+n)-\psi(\pi^{n-1}(j)\!+\!1\!-\alpha)>0,
\end{eqnarray*}
where the strict inequality follows from $\theta>-\alpha$ and $n>\hpi^n(j)-\alpha=\pi^{n-1}-\alpha+1$.
Now, let us prove it for the Gini index with $\kappa=1$. From (\ref{equ14}) and $(x+1)x-x(x-1)=2x$ we get
\begin{eqnarray*}
\Delta^{\rm \Gin}(\hpi^n)&=&(\theta+n)(\theta+n+1)-(\theta+n-1)(\theta+n)\\
&{}&-(\pi^{n-1}(j)-\alpha+1)(\pi^{n-1}(j)\!-\alpha\!+2)+(\pi^{n-1}(j)\!-\alpha)(\pi^{n-1}(j)\!-\alpha\!+1)\\
&=& 2((\theta+n)-(\pi^{n-1}(j)\!+\!1\!-\!\alpha))>0.
\end{eqnarray*}
Now let $\pi^n=\pi^n_+=(\pi^{n-1},1)$, so $k^n=k+1$. From (\ref{equ13}) and (\ref{equ17}) we obtain
\begin{eqnarray*}
\Delta^{\rm \mH}(\pi^n_+)&=&(\theta+n)\psi(\theta+n+1)-(\theta+n-1)\psi(\theta+n)\\
&{}&-(1-\alpha)\psi(2-\alpha)-\alpha \psi(1-\alpha)\\
&=&  \psi(\theta\!+n)-\psi(1-\alpha)>\Delta^{\rm \mH}_n(\hpi^n),
\end{eqnarray*}
the last strict inequality because $\pi^{n-1}(j)\ge 1$.
So, $\Delta^{\rm \mH}(\pi^n)$ attains its maximum at $\pi^n=\pi^n_+$. 
Analogously, from (\ref{equ14}) we find,
\begin{eqnarray*}
\Delta^{\rm \Gin}(\pi^n_+)&=&(\theta+n)(\theta+n+1)-(\theta+n-1)(\theta+n)
-(1-\alpha)(2-\alpha)-\alpha(1-\alpha)\\
&=&2((\theta+n)-(1-\alpha))>\Delta^{\rm \Gin}_n(\hpi^n),
\end{eqnarray*}
and so $\Delta^{\rm \Gin}(\pi^n)$ also attains its maximum at $\pi^n=\pi^n_+$. 
For the generalized Gini index with $\kappa\in \NN$, we have,
\begin{eqnarray*}
&{}&\Delta^{\rm{\Gin^{(\kappa)}}}(\hpi^n)=
(\kappa+1)\big(\prod_{r=0}^{\kappa-1}(\theta\!+\!n\!+r)-\prod_{r=0}^{\kappa-1}(\pi^n(j)\!+r\!-\alpha)\big)>0 \hbox{ and }\\
&{}&\Delta^{\rm{\Gin^{(\kappa)}}}(\pi^n_+)=(\kappa+1)\big(\prod_{r=0}^{\kappa-1}(\theta\!+\!n\!+r)-
\prod_{r=0}^{\kappa-1}(1\!+\!r\!-\alpha)\big)>\Delta^{\rm{\Gin^{(\kappa)}}}(\hpi^n).
\end{eqnarray*}
\end{proof}

\smallskip

${\bf IVb}$. Let us now state similar results as those stated in  ${\bf IVa}$. We use similar notation,
so $\pi^n\in \mN_n(k)$ is such that $\pi^n(j)\ge \pi^n(l)+2$ and we define $\opi^n\in \mN_n(k)$ 
with $\opi^n(i)=\pi^n(i)$ for $i\in \{1,..,k\}\setminus \{j,l\}$ and
$\opi^n(j)=\pi^n(j)-1$, $\opi^n(l)=\pi^n(l)+1$. Let us prove that,
\begin{equation}
\label{equ24}
E_{\bP(\cdot | \opi^n)}(\mH)\ge E_{\bP(\cdot | \pi^n)}(\mH) \hbox{ and }  
E_{\bP(\cdot | \opi^n)}(\Gin^{(\kappa)})\ge E_{\bP(\cdot | \pi^n)}(\Gin^{(\kappa)}).
\end{equation}
For the entropy it follows from (\ref{equ13}), by using (\ref{equ17})  
and that $\psi$ is increasing in $(0,\infty)$:
$$
(\theta+n) (E_{\bP(\cdot | \opi^n)}(\mH)-E_{\bP(\cdot | \pi^n)}(\mH))
=\psi(\pi^n(j)-\alpha)-\psi(\pi^n(l)+1-\alpha)\ge 0.
$$
For the Gini function with $\kappa=1$ it follows from (\ref{equ14}),
$$
(\theta+n) (\theta+n+1)(E_{\bP(\cdot | \opi^n)}(\Gin)-E_{\bP(\cdot | \pi^n)}(\Gin))
=2(\pi^n(j)-(\pi^n(l)+1))\ge 0.
$$
Fort the generalized Gini index with $\kappa\in \NN$ it follows from,
$$
\prod_{r=0}^\kappa \! (\theta\!+n\!+r)\big(\!E_{\bP(\cdot | \opi^n)}(\Gin^{(\kappa)})-\! E_{\bP(\cdot | \pi^n)}(\Gin^{(\kappa)})\!\big)
\!=\!(\kappa\!+1)\big(\!\prod_{r=0}^{\kappa-1}\! (\pi^n(j)\!+r\!-\!\alpha))-\!\prod_{r=0}^{\kappa-1}\! (\pi^n(l)\!+\!1+\!r\!-\!\alpha)\!\big).
$$

Therefore, by iterating (\ref{equ24}) one shows that:
\begin{eqnarray}
\label{equ25}
&{}&\;\;\;\;E_{\bP(\cdot | \pi^n_{\rm min(k)})}(\mH\!)\!=\!\!\!\!\min\limits_{\pi^n\in \mN_n(k)}\!E_{\bP(\cdot | \pi^n)}(\mH\!),\, 
\!E_{\bP(\cdot | \pi^n_{\rm max(k)})}(\mH\!)\!=\!\!\!\!\max\limits_{\pi^n\in \mN_n(k)}\!E_{\bP(\cdot | \pi^n)}(\mH\!),\\
\nonumber
\!\!\!\!\!\!&{}&\!\!\!\!\!E_{\bP(\cdot | \pi^n_{\rm min(k)})}(\!\Gin^{(\kappa)}\!)
\!=\!\!\!\!\min\limits_{\pi^n\in \mN_n(k)}\!E_{\bP(\cdot | \pi^n)}(\!\Gin^{(\kappa)}\!) ,\, 
\!E_{\bP(\cdot | \pi^n_{\rm max(k)})}(\!\Gin^{(\kappa)}\!)\!=\!\!\!\!\max\limits_{\pi^n\in \mN_n(k)}\!E_{\bP(\cdot | \pi^n)}(\!\Gin^{(\kappa)}\!),
\end{eqnarray}
where $\pi^n_{\rm min(k)}$ and $\pi^n_{\rm max(k)}$ are described at the end of IVa in Section \ref{sec5.2}. 
We mention that the equalities for the entropy in (\ref{equ25}) were firstly shown in Proposition 3.2 in \cite{martinez}.

\medskip

Let us use (\ref{equ25}) to prove Lemma \ref{lemma1}. It suffices to show that for $1\le k\le n-1$ one has:
\begin{equation}
\label{equ26}
\nonumber
E_{\bP(\cdot | \pi^n_{\rm min(k)})}(\mH)\le E_{\bP(\cdot | \pi^n_{\rm min(k+1)})} 
\hbox{ and }
E_{\bP(\cdot | \pi^n_{\rm min(k)})}(\Gin^{(\kappa)})\le E_{\bP(\cdot | \pi^n_{\rm min(k+1)})}(\Gin^{(\kappa)}).
\end{equation}
The first inequality follows from:
\begin{eqnarray*}
&{}&(\theta+n)\left(E_{\bP(\cdot | \pi^n_{\rm min(k)})}(\mH)- E_{\bP(\cdot | \pi^n_{\rm min(k+1)})}(\mH)\right)\\
&=&\alpha \psi(1\!-\!\alpha)\!+\!(1\!-\!\alpha)\psi(2\!-\!\alpha)\!-\!(n\!-k\!+1\!-\!\alpha) \psi(n\!-\!k\!+\!2\!-\!\alpha)
\!+\!(n\!-\!k\!-\!\alpha)\psi(n\!-k\!+\!1\!-\!\alpha)\\
&=&\psi(1-\alpha)-\psi(n-k+1-\alpha)< 0.
\end{eqnarray*}
Let us show the second one for the Gini index with $\kappa=1$ ,
\begin{eqnarray*}
&{}&(\theta+n)(\theta+n+1)\big(E_{\bP(\cdot | \pi^n_{\rm min(k)})}(\Gin)- E_{\bP(\cdot | \pi^n_{\rm min(k+1)})}(\Gin)\big)\\
&=&\alpha (1\!-\!\alpha)\!+\!(1\!-\!\alpha)(2\!-\!\alpha)\!-\!(n\!-k\!+1\!-\!\alpha) (n\!-\!k\!+\!2\!-\!\alpha)
\!+\!(n\!-\!k\!-\!\alpha)(n\!-k\!+\!1\!-\!\alpha)\\
&=&2((1-\alpha)-(n-k+1-\alpha))< 0.
\end{eqnarray*}
For the generalized Gini index with $\kappa\in \NN$ , the last expression becomes,
$$
\prod_{r=0}^{\kappa}(\theta+r)\big(E_{\bP(\cdot | \pi^n_{\rm min(k)})}(\Gin)- E_{\bP(\cdot | \pi^n_{\rm min(k+1)})}(\Gin)\big)
\!=\!(\kappa+1)\big(\prod_{r=1}^{\kappa}(r\!-\alpha)-\prod_{r=1}^{\kappa}(r\!-\alpha\!+n\!-k)\big)\!<\!0.
$$
So, Lemma \ref{lemma1} is proven.

\subsection{On the integrability of the R\'enyi entropy}
\label{sec5.4}
The R\'enyi entropy of parameter $\zeta>0$ and $\zeta\neq 1$, is defined by
$\mH^{\rm{R}}_\zeta(s)=\frac{1}{1-\zeta} \log \left(\sum_{i\in I} s(i)^\zeta \right)$, see \cite{renyi}.
The following equality is satisfied $\lim\limits_{\zeta\to 1} \mH^{\rm{R}}_\zeta(s)=\mH(s)$. 
When $0<\zeta<1$ the R\'enyi entropy satisfies hypotheses of an impurity function.
To be able to apply the martingale characterization and limit results of Proposition \ref{prop1},
and also for using  \cite{antos}, let us see its integrability on the exchangeable random partitions, 
mainly on the PDP. 

\smallskip

Let $\zeta>1$. So $\sum_{i\in \NN} s(i)^\zeta\le 1$, then $\mH^{\rm{R}}_\zeta$ is integrable if
$\int \log(\sum_{i\in \NN} \hpi(i)^\zeta) d\PP>-\infty$.
Since $\sum_{i\in \NN} \hpi(i)^\zeta\ge \hpi(1)^\zeta$, a sufficient condition 
for integrability is $\int \log(\hpi(1))d\PP>-\infty$, that is if 
$\int_0^1 {\log}(x)dF(x)\!> \!-\infty$ which is the integrability condition for 
the Shannon entropy $\mH$. So, for the PDP and when $\zeta>1$, 
the R\'enyi entropy $\mH^{\rm{R}}_\zeta$ is integrable and so Proposition \ref{prop1} can be applied. 

\medskip

Let $\zeta\in (0,1)$. So, $\sum_{i\in \NN} s(i)^\zeta >1$, then $H^{\rm{R}}_\zeta$ is integrable when
$\int \log(\sum_{i\in \NN} \hpi(i)^\zeta)d\PP<\infty$.  A sufficient 
condition to have integrability is $\int (\sum_{i\in \NN} \hpi(i)^\zeta)d\PP<\infty$, this is 
$\int_0^1 x^{\zeta-1}dF(x)<\infty$.
Then, for the PDP$(\alpha,\theta)$ a sufficient condition for the integrability
of $\mH^{\rm{R}}_\zeta$ is $\alpha<\zeta<1$.
We mention that under this condition, in addition to Proposition \ref{prop1}, 
the plug-in estimator of the R\'enyi entropy
satisfies $\mH^{\rm{R}}_\zeta=\lim\limits_{n\to \infty} \mH^{\rm{R}}_\zeta(\pi^n/n) \; \; 
\bP\!-\!\hbox{a.s.}$. See p. 171 in \cite{antos}. 

\bigskip

\noindent {\bf Acknowledgments.} This work was supported by the Center for 
Mathematical Modeling ANID Basal PIA program FB210005. The author thanks Jorge Silva from 
DIE University of Chile for discussions on the plug-in estimators of entropy for an infinite alphabet 
and for calling my attention to reference \cite{antos}.

\bigskip

Conflict of interest: The author declares that has no conflict of interest.

\medskip

Data Availability: Data sharing is not applicable to this article as not data sets were generated 
or analyzed during the current study.

\medskip

Declaration of generative AI and AI-assisted technologies: The author declares that has not used
this type of technologies during the current study.

\end{document}